\documentclass{amsart}

\usepackage{amsmath}
\usepackage{amssymb}
\usepackage[]{amsrefs}
\usepackage{xpatch}
\usepackage{kantlipsum} % for text filler
\calclayout

\xpatchcmd{\proof}{\itshape}{\prooflabelfont}{}{}
\newcommand{\prooflabelfont}{\bfseries}

\DefineSimpleKey{bib}{primaryclass}{}
\DefineSimpleKey{bib}{archiveprefix}{}

\BibSpec{arXiv}{%
  +{}{\PrintAuthors}{author}
  +{,}{ \textit}{title}
  +{}{ \parenthesize}{date}
  +{,}{ arXiv }{eprint}
  +{,}{ primary class }{primaryclass}
}

\usepackage{kantlipsum} % for text filler
\calclayout
\usepackage{extarrows}
\usepackage{amssymb}
\usepackage[utf8]{inputenc}
\newtheorem{theorem}{Theorem}[section]
\newtheorem{proposition}[theorem]{Proposition}
\newtheorem{lemma}[theorem]{Lemma}
\newtheorem{corollary}[theorem]{Corollary}
\theoremstyle{definition}

\theoremstyle{definition}
\newtheorem{definition}[theorem]{Definition}
\newtheorem{remark}[theorem]{Remark}

\numberwithin{equation}{section}
\usepackage{amsthm,amsmath,amssymb}
\usepackage[all,cmtip]{xy}
\usepackage{amsmath}
\usepackage{amssymb}
\usepackage{amsthm}
\usepackage[]{amsrefs}
\usepackage{hyperref}
\usepackage{xpatch}
\usepackage{amsfonts}
\usepackage{amssymb}
\usepackage[utf8]{inputenc}
\usepackage{amsthm}
\usepackage{stmaryrd}
\usepackage{csquotes}
\usepackage{extarrows}
\MakeOuterQuote{"}
\theoremstyle{definition}

\usepackage{enumitem}
\usepackage{tikz-cd}

\usepackage{blindtext}

\DeclareMathOperator{\Dim}{dim}
\DeclareMathOperator{\Ht}{height}
\DeclareMathOperator{\Depth}{depth}
\DeclareMathOperator{\Ext}{Ext}

\DeclareMathOperator{\CM}{CM}
\DeclareMathOperator{\MCM}{MCM}
\DeclareMathOperator{\SC}{SC}
\DeclareMathOperator{\Ann}{Ann}
\DeclareMathOperator{\Ass}{Ass}
\DeclareMathOperator{\Pic}{Pic}

\DeclareMathOperator{\Hom}{Hom}

\begin{document}

\title[The index of a numerical semigroup ring]{The index of a numerical semigroup ring}

\author[Richard F. Bartels]{Richard F. Bartels}

\address{Department of Mathematics, Trinity College, 300 Summit St, Hartford, CT, 06106}

\email{rbartels@trincoll.edu}

\urladdr{https://sites.google.com/view/richard-bartels-math/home}

\subjclass[2020]{13B30, 13C13, 13C14, 13C15, 13D02, 13D07, 13E05, 13E15, 13H05, 13H10.}

\keywords{Cohen-Macaulay, generically Gorenstein, canonical ideal, MCM approximation, FID hull}

\title{Cohen-Macaulay approximations and the $\SC_r$-condition}
\begin{abstract} We study the relation between MCM approximations and FID hulls of modules over a Cohen-Macaulay local ring $R$ with canonical module, specifically when $R$ is generically Gorenstein. We then generalize a result of Kato, who proved that a Gorenstein complete local ring $R$ satisfies the $\SC_{2}$-condition if and only if $R$ is a UFD. For $r \geq 3$, we prove a criterion for when an MCM $R$-module $M$ satisfies the $\SC_{r}$-condition, assuming that its first syzygy $\Omega_{R}^{1}(M)$ satisfies the $\SC_{r-1}$-condition.    
\end{abstract}
\maketitle
\large{
\begin{center}
\section{Introduction}
\end{center}
Let $(R,\mathfrak{m})$ be a Cohen-Macaulay local ring with canonical module $\omega$. The {\it{minimal MCM approximation}} of a finitely-generated $R$-module $M$ is an exact sequence of $R$-modules  
\[
0 \longrightarrow Y_M \xlongrightarrow{\iota} X_M \longrightarrow M \longrightarrow 0
\] where $Y_M$ has finite injective dimension, $X_M$ is MCM, and $Y_M$ and $X_M$ have no direct summand in common via $\iota$. The {\it{minimal FID hull}} of $M$ is an exact sequence of $R$-modules
\[
0 \longrightarrow M \longrightarrow Y^M \xlongrightarrow{\pi} X^M \longrightarrow 0
\] where $Y^M$ has finite injective dimension, $X^M$ is MCM or zero, and $Y^M$ and $X^M$ have no direct summand in common via $\pi$. Each finitely-generated $R$-module has a minimal MCM approximation and a minimal FID hull. These sequences are unique up to isomorphism of exact sequences inducing the identity on $M$ \cite[Definitions 11.8 and 11.10, Proposition 11.13, Theorem 11.17]{LW12}. 

A Cohen-Macaulay local ring $R$ with canonical module $\omega$ is {\it{generically Gorenstein}} if $R_{\mathfrak{p}}$ is a Gorenstein local ring for each minimal prime ideal $\mathfrak{p}$ of $R$. When $R$ is not Gorenstein, this condition is equivalent to $\omega$ being isomorphic to a height one ideal of $R$ (see \cite[Proposition 3.3.18]{BH93} and \cite[Proposition 11.6]{LW12}). Such an ideal is called a {\it canonical ideal} and also denoted $\omega$. In the following, we prove isomorphisms relating MCM approximations and FID hulls of modules over a generically Gorenstein ring that are obtained from ideals containing $\omega$ and contained in $\omega$ (Propositions \ref{prop:2.6} and \ref{prop:2.7}). In Proposition \ref{prop:2.8}, we use these results to prove the following: Let $R$ be a Cohen-Macaulay local ring with canonical module that is generically Gorenstein. Let $\omega \subseteq R$ be a canonical ideal and $x \in \omega$ an $R$-regular element. Let $(\omega/xR)^{\vee}:=\text{Ext}_{R}^{1}(\omega/xR,\omega)$. Then $\omega/xR$ and $(\omega/xR)^{\vee}$ are Cohen-Macaulay $R$-modules of codimension 1, and up to adding or deleting direct summands isomorphic to $\omega$, we have the following isomorphisms.
\[
X_{\omega/xR} \cong X_{(\omega/xR)^{\vee}} \cong X^{R/\omega}.
\] We have the following exact sequence characterizing $X^{R/\omega}$, where $n=\mu_{R}(\omega)$.
\[
0 \longrightarrow R \longrightarrow \omega^{n} \longrightarrow X^{R/\omega} \longrightarrow 0.
\] 

In section 3, we study conditions for when an MCM module $C$ is an MCM approximation of a finitely-generated module $M$ of some fixed codimension (i.e., $C \cong X_{M}$). This work is motivated in part by the following uniqueness result for FID hulls: For a Gorenstein complete local ring $R$, Kato proved that if $M$ and $N$ are finitely-generated $R$-modules such that $M$ has positive codimension, $X^M \cong X^N$, and $Y^M \cong Y^N$, then $M \cong N$ \cite[Theorem 1.2]{KK07}.

Since MCM approximations and FID hulls are dual constructions, it is natural to ask if the map $M \mapsto X_M$ from finitely-generated modules of positive codimension to isomorphism classes of MCM modules is surjective. This leads to the following definition: Let $R$ be a $d$-dimensional Cohen-Macaulay local ring with canonical module and let $C$ be an MCM $R$-module. For $0 \leq r \leq d$, we say that $C$ satisfies the {\it{$\SC_r$-condition}} if $C$ is {\it{stably isomorphic}} to the minimal MCM approximation of a finitely-generated $R$-module of codimension $r$. That is, for some free $R$-modules $F$ and $G$ and an $R$-module $M$ of codimension $r$, we have $C \oplus F \cong X_M \oplus G$. In this case, we write $C \overset{st}{\cong} X_M$. If $C$ satisfies the $\SC_{r}$-condition, then we can assume $M$ is a Cohen-Macaulay $R$-module \cite[Proposition 2.2]{KK07}. If each MCM $R$-module satisfies the $\SC_r$-condition, we say $R$ satisfies the $\SC_r$-condition. 

Our study of the $\SC_r$-condition is also motivated by the following result of Kato: If $R$ is a complete Gorenstein local ring that satisfies the $\SC_r$-condition for some positive integer $r$, then the localization $R_{\mathfrak{p}}$ is regular for each prime ideal $\mathfrak{p}$ of $R$ of height less than $r$ \cite[Proposition 2.5]{KK07}.

Since every MCM module is its own minimal MCM approximation, every Cohen-Macaulay local ring with canonical module satisfies the $\SC_{0}$-condition. If an MCM module satisfies the $\SC_{r+1}$-condition for some $r>0$, then it also satisfies the $\SC_{r}$-condition. Therefore, the classes of rings which satisfy the $\SC_{r}$-conditions for $r \geq 0$ are ordered by inclusion \cite[Proposition 2.5]{KK07}. 

Kato proved that a Gorenstein complete local ring $R$ satisfies the $\SC_1$-condition if and only if $R$ is a domain \cite[Theorem 3.3]{KK07}. More generally, Leuschke and Weigand proved that if $R$ is a Cohen-Macaulay local ring with canonical module that is generically Gorenstein, then $R$ satisfies the $\SC_1$-condition if and only if $R$ is a domain \cite[Corollary 11.23]{LW12}. Yoshino and Isogawa proved that the following conditions are equivalent for a normal Gorenstein complete local ring $R$ of dimension 2 \cite[Theorem 2.2]{YI00}:\newline
\begin{enumerate}[label=(\alph*)]
\item $R$ is a UFD. \newline 
\item For any MCM $R$-module, there is an $R$-module $L$ of finite length (hence, a Cohen-Macaulay $R$-module of codimension 2) such that $M \overset{st}{\cong} \Omega_{R}^{2}(L)$.
\newline 
\item $R$ satisfies the $\SC_2$-condition.
\newline
\end{enumerate}
Kato generalized this result, proving that a complete Gorenstein local ring $R$ satisfies the $\SC_2$-condition if and only if $R$ is a UFD \cite[Theorem 2.9]{KK07}.

Now, let $R$ be a Gorenstein complete local ring of dimension $d \geq 3$. For $r>0$, let $\CM^r(R)$ denote the class of all Cohen-Macaulay $R$-modules of codimension $r$, and let $\CM(R)$ denote the class of $\MCM$ $R$-modules. For $3 \leq r \leq d$, we prove the following criterion for when an MCM $R$-module $M$ satisfies the $\SC_r$-condition (Proposition \ref{prop:3.12}): Let $M$ be an MCM $R$-module and suppose $\Omega_{R}^{1}(M)$ satisfies the $\SC_{r-1}$-condition. Let $L \in \CM^{r-1}(R)$ such that $X_{L} \stackrel{st}\cong \Omega_{R}^1(M)$. If there is a regular sequence \,${\bf{x}} \in \Ann_{R}(L)$\, of length $r-2$ such that $R/{\bf{x}}R$ is a UFD, then $M$ satisfies the $\SC_{r}$-condition.

In Corollary \ref{corollary:3.16}, we use this criterion to prove the equivalence of the $\SC_d$- and $\SC_{d-1}$-conditions for Gorenstein complete local rings of dimension $d \geq 3$ that remain UFDs after factoring out certain regular sequences of length $d-2$.
\newline
\section{MCM approximations and FID hulls over generically Gorenstein rings}
Let $(R,\mathfrak{m})$ be a Cohen-Macaulay local ring with canonical module $\omega$. For every finitely-generated $R$-module $M$, there is an exact sequence 
\[
0 \longrightarrow Y \xlongrightarrow{\iota} X \longrightarrow M \longrightarrow 0
\] with $Y$ of finite injective dimension and $X$ an MCM $R$-module, called an MCM approximation of $M$. If $Y$ and $X$ have no direct summand in common via $\iota$, then the MCM approximation is {\it{minimal}}, and denoted as follows.
\[
0 \longrightarrow Y_M \longrightarrow X_M \longrightarrow M \longrightarrow 0
\] Dually, there is an exact sequence of $R$-modules 
\[
0 \longrightarrow M \longrightarrow Y' \xlongrightarrow{\pi} X' \longrightarrow 0
\] with $Y'$ of finite injective dimension and $X'$ either MCM or zero, called an FID hull of $M$. If $Y'$ and $X'$ have no direct summand in common via $\pi$, then the FID hull is minimal, and denoted as follows.
\[
0 \longrightarrow M \longrightarrow Y^M \longrightarrow X^M \longrightarrow 0
\] Each finitely-generated $R$-module has a minimal MCM approximation and a minimal FID hull. These sequences are unique up to isomorphism of exact sequences inducing the identity on $M$ \cite[Proposition 11.13]{LW12}. For an MCM approximation $0 \longrightarrow Y \xlongrightarrow{\iota} X \longrightarrow M \longrightarrow 0$, if $Y$ and $X$ have a direct summand $N$ in common via $\iota$, then $N$ is MCM and of finite injective dimension. Therefore, $N \cong \omega^m$ for some positive integer $m$. For an FID hull $0 \longrightarrow M \longrightarrow Y' \xlongrightarrow{\pi} X' \longrightarrow 0$, if $Y'$ and $X'$ have a direct summand $N'$ in common via $\pi$, then $N' \cong \omega^n$ for some positive integer $n$ \cite[Proposition 11.7]{LW12}. As a result, an arbitrary MCM approximation and an arbitrary FID hull can be written as follows.
\begin{proposition}\cite[Propositions 1.5 and 1.6]{Ding90}
Let $(R,\mathfrak{m})$ be a Cohen-Macaulay local ring with canonical module $\omega$. Let $M$ be a finitely-generated $R$-module. An arbitrary MCM approximation of $M$ can be written as follows for some non-negative integer $m$.
\[
0 \longrightarrow \omega^{m} \oplus Y_{M} \longrightarrow \omega^{m} \oplus X_M \longrightarrow M \longrightarrow 0
\] Likewise, an arbitrary FID hull of $M$ can be written as follows for some non-negative integer $n$.
\[
0 \longrightarrow M \longrightarrow \omega^{n} \oplus Y^{M} \longrightarrow \omega^{n} \oplus X^{M} \longrightarrow 0
\]
\end{proposition} The following definition gives the notation we use for modules that are isomorphic up to adding or deleting copies of the canonical module. 

\begin{definition} Let $(R,\mathfrak{m})$ be a Cohen-Macaulay local ring with canonical modules $\omega$. We say that two $R$-modules $M$ and $N$ are $\omega$-{\it{stably isomorphic}}, and write $M \cong_{\omega} N$, if for some non-negative integers $s$ and $t$, we have 
\[
M \oplus \omega^{s} \cong N \oplus \omega^{t}.
\]
\end{definition}

\begin{remark}
Suppose $(R,\mathfrak{m})$ is a Cohen-Macaulay local ring with canonical module $\omega$ that is generically Gorenstein. If two $R$-modules $M$ and $N$ are $\omega$-stably isomorphic, then $M_{\mathfrak{p}}$ and $N_{\mathfrak{p}}$ are stably isomorphic $R_{\mathfrak{p}}$-modules for each minimal prime ideal $\mathfrak{p}$ of $R$.
\end{remark}

\begin{proposition}\cite[Propositions 11.17 and 11.19]{LW12}\label{prop:2.4} Let $R$ be a Cohen-Macaulay local ring with canonical module and $M$ a finitely-generated $R$-module. Let $F$ be a free cover of $M$ and $\Omega_{R}^{1}(M)$ the first syzygy of $M$. Then we have the following.\newline
\begin{enumerate}[label=(\alph*)]
\item $Y_{M} \cong_{\omega} Y^{\Omega_{R}^{1}(M)}$ \newline
\item $X^{M} \cong_{\omega} X^{X_M}$ \newline
\item There is a short exact sequence 
\[
0 \longrightarrow F \longrightarrow  \omega^m \oplus X_M \longrightarrow X^{\Omega_{R}^{1}(M)} \longrightarrow 0
\] for some non-negative integer $m$.\\\\
\end{enumerate}\text{}
If $R$ is Gorenstein, then we also have the following.\newline
\begin{enumerate}
\item[(d)] $X_{M} \cong_{\omega} X^{\Omega_{R}^{1}(M)}$ \newline
\item[(e)] $X_{M} \cong_{\omega} \Omega_{R}^{1}(X^M)$ \newline 
\item[(f)] $Y_{M} \cong_{\omega} \Omega_{R}^{1}(Y^M)$
\newline
\end{enumerate}
\end{proposition}
\begin{remark}
Parts (b), (e), and (f) of Proposition \ref{prop:2.4} can be deduced from the following commutative diagram with exact rows and columns \cite[Diagram 11.3]{LW12}. In the diagram, $n=\mu_{R}\left(\Hom_{R}(X_M,\omega) \right)$.
\begin{equation}\label{diagram:2.1}
\xymatrix{& 0 \ar[d] & 0 \ar[d] \\
& Y_{M} \ar@{=}[r] \ar[d] & Y_{M} \ar[d]\\  
0 \ar[r] &
X_{M} \ar[d] \,\ar[r] & \omega^{n} \ar[r] \ar[d] & X^M \ar@{=}[d] \ar[r]  & 0 \\
0 \ar[r] & M \ar[r] \ar[d] & Y^M \ar[r] \ar[d] & X^M \ar[r] & 0 \\
& 0 & 0 \\
}
\end{equation}
\end{remark}
We now prove a dual result to Proposition \ref{prop:2.4} (a)-(c).
\begin{proposition}
Let $(R,\mathfrak{m})$ be a Cohen-Macaulay local ring with canonical module and $M$ a finitely-generated $R$-module. Let $E_{R}(M)$ be the injective hull of $M$ and $\Omega_{1}^{R}(M)$ the first cosyzygy of $M$. Then we have the following.
\begin{enumerate}[label=(\alph*)] 
\item $X^M \cong_{\omega} X_{\Omega_{1}^{R}(M)}$
\newline 
\item $Y_M \cong Y_{Y^M}$
\newline
\item There is a short exact sequence 
\[
0 \longrightarrow Y_{\Omega_{1}^{R}(M)}  \longrightarrow \omega^n \oplus Y^M \longrightarrow E_{R}(M) \longrightarrow 0
\] for some non-negative integer $n$.
\end{enumerate}\text{}\\
\end{proposition}
\begin{proof}
We first prove (a) and (c). Consider the pullback diagram for the minimal injective resolution of $M$ and the minimal MCM approximation of $\Omega_{1}^{R}(M)$.

\begin{equation}\label{diagram:2.1.1}
\xymatrix{&& 0 \ar[d] & 0 \ar[d] \\
&& M \ar@{=}[r] \ar[d] & M \ar[d]\\  
0 \ar[r] &
Y_{\Omega_{1}^{R}(M)} \ar@{=}[d] \,\ar[r] & Z \ar[r] \ar[d] & E_{R}(M) \ar[d] \ar[r]  & 0 \\
0 \ar[r] & Y_{\Omega_{1}^{R}(M)} \ar[r] & X_{\Omega_{1}^{R}(M)} \ar[r] \ar[d] & \Omega_{1}^{R}(M) \ar[r] \ar[d]& 0 \\
&& 0 & 0 \\
} 
\end{equation}
From the middle row of diagram \ref{diagram:2.1.1}, we see that $Z$ has finite injective dimension. Therefore, the middle column of the diagram 
\begin{equation}
0 \longrightarrow M \longrightarrow Z \longrightarrow X_{\Omega_{1}^{R}(M)} \longrightarrow 0 
\end{equation} is an FID hull of $M$. So $X_{\Omega_{1}^{R}(M)} \cong \omega^{n} \oplus X^M$ and $Z \cong \omega^{n} \oplus Y^M$ for some non-negative integer $n$. The middle row of diagram \ref{diagram:2.1.1} gives us the exact sequence in part (c). For part (b), consider the following exact sequence from the middle column of diagram \ref{diagram:2.1}. 
\begin{equation}\label{sequence:2.8}
0 \longrightarrow Y_M \longrightarrow w^n \longrightarrow Y^M \longrightarrow 0
\end{equation} 
Sequence \ref{sequence:2.8} is an MCM approximation of $Y^M$, so we have $Y_M \cong \omega^s \oplus Y_{Y^M}$ for some non-negative integer $s$. Now consider the pullback diagram for the minimal FID hull of $M$ and the minimal MCM approximation of $Y^M$.
\begin{equation}\label{diagram:2.5}
\xymatrix{& 0 \ar[d] & 0 \ar[d] \\
& Y_{Y^M} \ar@{=}[r] \ar[d] & Y_{Y^M} \ar[d]\\  
0 \ar[r] &
Z \ar[d] \,\ar[r] & X_{Y^M} \ar[r] \ar[d] & X^M \ar@{=}[d] \ar[r]  & 0 \\
0 \ar[r] & M \ar[r] \ar[d] & Y^M \ar[r] \ar[d] & X^M \ar[r] & 0 \\
& 0 & 0 \\
}
\end{equation} From the middle row of diagram \ref{diagram:2.5}, we see that $Z$ is an MCM $R$-module. Therefore, the first column gives us $Y_{Y^M} \cong \omega^t \oplus Y_M$ for some non-negative integer $t$. Since $Y_M \cong \omega^s \oplus Y_{Y^M}$, it folllows that 
\[
Y_{Y^M} \cong \omega^{s+t} \oplus Y_{Y^M}.
\] Therefore, $s=t=0$ and $Y_{M} \cong Y_{Y^M}$.
\end{proof}

\begin{lemma}\label{lemma:2.7}
Let $(R,\mathfrak{m})$ be a Cohen-Macaulay local ring. If
\begin{equation}\label{sequence:2.10.1}
0 \longrightarrow Y \longrightarrow M \longrightarrow X \longrightarrow 0
\end{equation} is an exact sequence, $X$ is an MCM $R$-module, and $Y$ is an $R$-module of finite injective dimension, then the sequence splits and $M \cong Y \oplus X$.
\end{lemma}
\begin{proof}
Applying $\Hom_{R}(-,Y)$, we obtain the following exact sequence.
\begin{equation*}
\Hom_{R}(M,Y) \longrightarrow \Hom_{R}(Y,Y) \longrightarrow \Ext_{R}^{1}(X,Y). 
\end{equation*} Since $X$ is MCM and $Y$ has finite injective dimension, it follows that $\Ext_{R}^{1}(X,Y)=0$ \cite[Theorem 11.3]{LW12}. Therefore, sequence \ref{sequence:2.10.1} splits. 
\end{proof}

\begin{proposition}\label{prop:2.6}
Let $(R,\mathfrak{m})$ be a Cohen-Macaulay local ring with canonical module that is generically Gorenstein. Let $\omega \subseteq R$ be a canonical ideal and let $\omega \subseteq I$ be an ideal of $R$. Then we have the following.\newline
\begin{enumerate}[label=(\alph*)]
\item $X_{I/\omega} \cong_{\omega} X_{I}$ \newline 
\item $Y_{I/\omega}\, \cong_{\omega} Y_{I}$ \newline
\item We have 
\[
X_I \cong \omega \oplus X_{I/\omega} \,\,\,\,\,\, \text{and} \,\,\,\,\,\, Y_{I/\omega} \cong Y_I
\] \,\,\,\,\,\,\,\,\,\,\,\,\,\,\,\,or 
\[
X_I \cong X_{I/\omega} \,\,\,\,\,\, \text{and} \,\,\,\,\,\, Y_{I/\omega} \cong \omega \oplus Y_I.
\]
\item $X^{I} \, \cong_{\omega} \, X^{I/\omega}$
\end{enumerate}\text{}
\end{proposition}
\begin{proof} Consider the pullback diagram for the minimal MCM approximation for $I/\omega$ and the sequence $0 \longrightarrow \omega \longrightarrow I \longrightarrow I/\omega \longrightarrow 0$.

\begin{equation}
\label{diagram:2.10.1}
\xymatrix{&& 0 \ar[d] & 0 \ar[d] \\
&& \omega \ar@{=}[r] \ar[d] & \omega \ar[d]\\  
0 \ar[r] &
Y_{I/\omega} \ar@{=}[d] \,\ar[r] & Z \ar[r] \ar[d] & I \ar[d] \ar[r]  & 0 \\
0 \ar[r] & Y_{I/\omega} \ar[r] & X_{I/\omega} \ar[r] \ar[d] & I/\omega \ar[r] \ar[d]& 0 \\
&& 0 & 0 \\
}
\end{equation}
The middle column of diagram \ref{diagram:2.10.1} is the exact sequence 
\begin{equation}\label{sequence:2.12}
0 \longrightarrow \omega \longrightarrow Z \longrightarrow X_{I/\omega} \longrightarrow 0.
\end{equation} By lemma \ref{lemma:2.7}, sequence \ref{sequence:2.12} splits and $Z \cong \omega \oplus X_{I/\omega}$. Therefore, the middle row of \ref{diagram:2.10.1} gives us the following exact sequence.
\begin{equation}\label{sequence:2.13.1}
0 \longrightarrow Y_{I/\omega} \longrightarrow \omega \oplus X_{I/\omega} \longrightarrow I \longrightarrow 0
\end{equation} Sequence \ref{sequence:2.13.1} is an MCM approximation of $I$. Therefore, for some non-negative integer $m$, we have $Y_{I/\omega} \cong \omega^m \oplus Y_I$ \,and\, $\omega \oplus X_{I/\omega} \cong \omega^m \oplus X_I$. Now consider the pullback diagram for the sequence
$
0 \longrightarrow \omega \longrightarrow I \longrightarrow I/\omega \longrightarrow 0$ and the minimal MCM approximation of $I$.

\begin{equation}\label{diagram:2.6}
\xymatrix{& 0 \ar[d] & 0 \ar[d] \\
& Y_{I} \ar@{=}[r] \ar[d] & Y_{I} \ar[d]\\  
0 \ar[r] &
Z \ar[d] \,\ar[r] & X_{I} \ar[r] \ar[d] & I/\omega \ar@{=}[d] \ar[r]  & 0 \\
0 \ar[r] & \omega \ar[r] \ar[d] & I \ar[r] \ar[d] & I/\omega \ar[r] & 0 \\
& 0 & 0 \\
} 
\end{equation}\\ By lemma \ref{lemma:2.7}, the first column of diagram \ref{diagram:2.6} splits. Therefore, $Z \cong \omega \oplus Y_I$ and the middle row of diagram \ref{diagram:2.6}
\begin{equation}\label{sequence:2.16.1}
0 \longrightarrow \omega \oplus Y_I \longrightarrow X_{I} \longrightarrow I/\omega \longrightarrow 0
\end{equation} is an MCM approximation of $I/\omega$. It follows that $\omega \oplus Y_{I} \cong \omega^n \oplus Y_{I/\omega}$ and $X_{I} \cong \omega^n \oplus X_{I/\omega}$ for some non-negative integer $n$. Since $\omega \oplus X_{I/\omega} \cong \omega^m \oplus X_I$, we have
\[
\omega \oplus X_{I/\omega} \cong \omega^{m+n} \oplus X_{I/\omega}.
\] Therefore, $m+n=1$. We either have $m=0$ and $n=1$ or $m=1$ and $n=0$. This gives us parts (a) through (c). For part (d), consider the pushout diagram for
$
0 \longrightarrow \omega \longrightarrow I \longrightarrow I/\omega \longrightarrow 0$ and the minimal FID hull of $I$.
\begin{equation}\label{diagram:2.7}
\xymatrix{&& 0 \ar[d] & 0 \ar[d] \\
0 \ar[r]& \omega \ar[r] \ar@{=}[d] & I \ar[r] \ar[d] & I/\omega \ar[d] \ar[r] & 0\\  
0 \ar[r] &
\omega \,\ar[r] & Y^I \ar[r] \ar[d] & Z \ar[d] \ar[r]  & 0 \\
&& X^I \ar@{=}[r] \ar[d] & X^I \ar[d] \\
&& 0 & 0 \\
} 
\end{equation} Since $\omega$ and $Y^I$ have finite injective dimension, $Z$ also has finite injective dimension. Therefore, the last column of diagram \ref{diagram:2.7} 
\begin{equation*}
0 \longrightarrow I/\omega \longrightarrow Z \longrightarrow X^I \longrightarrow 0
\end{equation*}\newline is an FID hull of $I/\omega$ and $X^I \cong_{\omega}X^{I/\omega}$.
\end{proof}\text{}
\begin{proposition}\label{prop:2.7}
Let $(R,\mathfrak{m})$ be a Cohen-Macaulay local ring with canonical module that is generically Gorenstein. Let $\omega \subseteq R$ be a canonical ideal and let $I \subseteq \omega$ be an ideal of $R$. Let $E_{R}(\omega)$ be the injective hull of $\omega$. Then we have the following.\newline
\begin{enumerate}[label=(\alph*)]
\item $ X^{I} \cong_{\omega} X_{\omega/I}$
\newline
\item  $Y^{I} \cong_{\omega} Y_{\omega/I}$  \newline 
\item We have 
\[
X_{\omega/I} \cong \omega \oplus X^{I} \,\,\,\,\,\, \text{and} \,\,\,\,\,\, Y^I \cong Y_{\omega/I}
\] \,\,\,\,\,\,\,\,\,\,\,\,\,\,\,\,or 
\[
X_{\omega/I} \cong X^{I} \,\,\,\,\,\, \text{and} \,\,\,\,\,\, Y^{I} \cong \omega \oplus Y_{\omega/I}.
\]
\item $X^{\omega/I} \cong_{\omega} X^{E_{R}(\omega)/I}$\newline
\item $X_{\omega/I} \cong_{\omega} X_{E_{R}(\omega)/I}$ \newline
\end{enumerate} 
\end{proposition}
\begin{proof} Consider the pushout diagram for the sequence $0 \longrightarrow I \longrightarrow \omega \longrightarrow \omega/I \longrightarrow 0$ and the minimal FID hull for $I$.

\begin{equation}\label{diagram:2.9}
\xymatrix{&& 0 \ar[d] & 0 \ar[d] \\
&0 \ar[r]& I \ar[r] \ar[d] & \omega \ar[r] \ar[d] & \omega/I \ar@{=}[d] \ar[r] & 0\\  
&0 \ar[r] &
Y^I \,\ar[r] \ar[d] & Z \ar[r] \ar[d] & \omega/I \ar[r]  & 0 \\
&& X^I \ar@{=}[r] \ar[d] & X^I \ar[d] \\
&& 0 & 0 \\
}
\end{equation}
From the middle column of diagram \ref{diagram:2.9}, we obtain the exact sequence
\begin{equation}\label{sequence:2.10}
0 \longrightarrow \omega \longrightarrow Z \longrightarrow X^I \longrightarrow 0.
\end{equation} This sequence splits by lemma \ref{lemma:2.7}.  The middle row of diagram \ref{diagram:2.9} then gives us the following exact sequence.
\begin{equation}\label{sequence:2.11}
0 \longrightarrow Y^I \longrightarrow \omega \oplus X^I \longrightarrow \omega/I \longrightarrow 0
\end{equation} Therefore, for some non-negative integer $m$, we have $Y^I \cong \omega^m \oplus Y_{\omega/I}$ and $\omega \oplus X^I \cong \omega^m \oplus X_{\omega/I}$. Now consider the pullback diagram for $0 \longrightarrow I \longrightarrow \omega \longrightarrow \omega/I \longrightarrow 0$ and the minimal MCM approximation of $\omega/I$.

\begin{equation}
\label{diagram:2.22}
\xymatrix{&& 0 \ar[d] & 0 \ar[d] \\
&& Y_{\omega/I} \ar@{=}[r] \ar[d] & Y_{\omega/I} \ar[d]\\  
0 \ar[r] &
I \ar@{=}[d] \,\ar[r] & Z \ar[r] \ar[d] & X_{\omega/I} \ar[d] \ar[r]  & 0 \\
0 \ar[r] & I \ar[r] & \omega \ar[r] \ar[d] & \omega/I \ar[r] \ar[d]& 0 \\
&& 0 & 0 \\
}
\end{equation}
Since the middle column splits by lemma \ref{lemma:2.7}, the middle row gives us the exact sequence
\begin{equation}\label{sequence:2.26}
0 \longrightarrow I \longrightarrow \omega \oplus Y_{\omega/I} \longrightarrow X_{\omega/I} \longrightarrow 0.
\end{equation}
Sequence \ref{sequence:2.26} is an FID hull of $I$. It follows that $\omega \oplus Y_{\omega/I} \cong \omega^n \oplus Y^I$ and $X_{\omega/I} \cong \omega^n \oplus X^I$ for some non-negative integer $n$. Since we also have $\omega \oplus X^I \cong \omega^m \oplus X_{\omega/I}$, it follows that
\[
\omega \oplus X^{I} \cong \omega^{m+n} \oplus  X^{I}.
\]
Therefore, $m+n=1$. So either $m=0$ and $n=1$ or $m=1$ and $n=0$. This gives us parts (a) through (c). For part (d), consider the pushout diagram for 
\[
0 \longrightarrow I \longrightarrow \omega \longrightarrow \omega/I \longrightarrow 0
\]
and the minimal injective resolution of $\omega$.
\begin{equation}\label{diagram:2.12}
\xymatrix{
&& 0  \ar[d] & 0 \ar[d]\\  
0 \ar[r] &
I \ar@{=}[d] \,\ar[r] & \omega \ar[r] \ar[d] & \omega/I \ar[d] \ar[r]  & 0 \\
0 \ar[r] & I \ar[r] & E_{R}(\omega) \ar[r] \ar[d] & Z \ar[r] \ar[d]& 0 \\
&& \Omega_{1}^{R}(\omega) \ar[d] \ar@{=}[r] & \Omega_{1}^{R}(\omega) \ar[d] \\
&& 0  & 0 \\
} 
\end{equation} 
From the middle row of diagram \ref{diagram:2.12}, we have $Z \cong E_{R}(\omega)/I$. Therefore, the last column of diagram \ref{diagram:2.12} gives us the exact sequence 
\begin{equation}\label{sequence:2.13}
0 \longrightarrow \omega/I \longrightarrow E_{R}(\omega)/I \longrightarrow \Omega_{1}^{R}(\omega) \longrightarrow 0.
\end{equation} Now consider the pushout diagram for sequence \ref{sequence:2.13} and the minimal FID hull of $\omega/I$.

\begin{equation}\label{diagram:2.14}
\xymatrix{&& 0 \ar[d] & 0 \ar[d] \\
&0 \ar[r]& \omega/I \ar[r] \ar[d] & Y^{\omega/I} \ar[r] \ar[d] & X^{\omega/I} \ar@{=}[d] \ar[r] & 0\\  
&0 \ar[r] &
E_{R}(\omega)/I \,\ar[r] \ar[d] & Z' \ar[r] \ar[d] & X^{\omega/I} \ar[r]  & 0 \\
&& \Omega_{1}^{R}(\omega) \ar@{=}[r] \ar[d] & \Omega_{1}^{R}(\omega) \ar[d] \\
&& 0 & 0 \\
}
\end{equation} From the middle column of diagram \ref{diagram:2.14}, we see that $Z'$ has finite injective dimension. The middle row is the exact sequence
\begin{equation}\label{sequence:2.15}
0 \longrightarrow E_{R}(\omega)/I \longrightarrow Z' \longrightarrow X^{\omega/I} \longrightarrow 0.
\end{equation} Sequence \ref{sequence:2.15} is an FID hull for $E_{R}(\omega)/I$, so $X^{\omega/I} \cong_{\omega} X^{E_{R}(\omega)/I}$. For part (e), consider the pushout diagram for the sequence $0 \longrightarrow I \longrightarrow E_{R}(\omega) \longrightarrow E_{R}(\omega)/I \longrightarrow 0$ and the minimal FID hull of $I$.

\begin{equation}\label{diagram:2.27}
\xymatrix{&& 0 \ar[d] & 0 \ar[d] \\
&0 \ar[r]& I \ar[r] \ar[d] & Y^{I} \ar[r] \ar[d] & X^{I} \ar@{=}[d] \ar[r] & 0\\  
&0 \ar[r] &
E_{R}(\omega) \,\ar[r] \ar[d] & Z \ar[r] \ar[d] & X^{I} \ar[r]  & 0 \\
&& E_{R}(\omega)/I \ar@{=}[r] \ar[d] & E_{R}(\omega)/I \ar[d] \\
&& 0 & 0 \\
}
\end{equation}
Since the middle row of diagram \ref{diagram:2.27} splits, the middle column gives us the exact sequence 
\begin{equation}\label{sequence:2.28}
0 \longrightarrow Y^I \longrightarrow E_{R}(\omega) \oplus X^I \longrightarrow E_{R}(\omega)/I \longrightarrow 0.
\end{equation} Consider the pullback diagram for sequence \ref{sequence:2.28} and the exact sequence
\[
0 \longrightarrow Y_{E_R(\omega)} \longrightarrow X_{E_R(\omega)} \oplus X^I \longrightarrow E_{R}(\omega)  \oplus X^I \longrightarrow 0.
\]

\begin{equation}\label{diagram:2.29}
\xymatrix{&& 0 \ar[d] & 0 \ar[d] \\
&& Y_{E_{R}(\omega)} \ar@{=}[r] \ar[d]& Y_{E_{R}(\omega)} \ar[d] \\
&0 \ar[r]& Z' \ar[r] \ar[d] & X_{E_{R}(\omega)} \oplus X^{I} \ar[r] \ar[d] & E_{R}(\omega)/I \ar@{=}[d] \ar[r] & 0\\  
&0 \ar[r] &
Y^I \,\ar[r] \ar[d] & E_{R}(\omega) \oplus X^I \ar[r] \ar[d] & E_{R}(\omega)/I \ar[r]  & 0 \\
&& 0 & 0 \\
} 
\end{equation}
From the first column of diagram \ref{diagram:2.29}, we see that $Z'$ has finite injective dimension. Since $X_{E_{R}(\omega)}$ has finite injective dimension, we have $X_{E_{R}(\omega)} \cong \omega^s$ for some non-negative integer $s$. Therefore, the middle row is an MCM approximation of $E_{R}(\omega)/I$ and $X^I \cong_{\omega} X_{E_{R}(\omega)/I}$. By part (a), we have $X_{\omega/I} \cong_{\omega} X_{E_{R}(\omega)/I}$.
\end{proof}\text{}
\begin{proposition}\label{prop:2.8}
Let $(R,\mathfrak{m})$ be a Cohen-Macaulay local ring of dimension $d$ that is generically Gorenstein and not Gorenstein. Let $\omega$ be a canonical ideal of $R$ and $x \in \omega$ an $R$-regular element. Then we have the following.\newline
\begin{enumerate}[label=(\alph*)]
\item $\omega/xR$ is a Cohen-Macaulay $R$-module of codimension $1$ \newline 
\item $\left(\omega/xR \right)^{\vee}:=\Ext_{R}^{1}(\omega/xR, \omega)$ is a Cohen-Macaulay $R$-module of codimension $1$
\newline 
\item $X_{\omega/xR} \cong_{\omega} X_{\left(\omega/xR \right)^{\vee}} \cong_{\omega} X^{R/\omega}$ \newline 
\item There is an exact sequence 
\[
0 \longrightarrow R \longrightarrow \omega^{n} \longrightarrow X^{R/\omega} \longrightarrow 0
\] with\, $n=\mu_{R}(\omega)$.
\end{enumerate}
\end{proposition}
\begin{proof}
Consider the sequence 
\begin{equation}\label{sequence:2.16}
0 \longrightarrow \omega/xR \longrightarrow R/xR \longrightarrow R/\omega \longrightarrow 0.
\end{equation} We have
\[
\text{dim}_{R}(R/xR)=\text{max}\{\text{dim}_{R}(\omega/xR),\,\, \text{dim}_{R}(R/\omega) \},
\] so\, $\text{dim}_{R}(\omega/xR) \leq d-1$. By \cite[Corollary 2.1.4 and Proposition 3.3.18]{BH93}, we have $\Dim_{R}(R/\omega) = d-1$. Applying this equality and the depth lemma to the sequence
\[
0 \longrightarrow \omega \longrightarrow R \longrightarrow R/\omega \longrightarrow 0,
\] we obtain $\Depth_{R}(R/\omega)=d-1$. Applying this equality, the depth lemma, and the inequality $\Dim_{R}(\omega/xR) \leq d-1$ to sequence \ref{sequence:2.16}, we obtain $\text{dim}_{R}(\omega/xR)=\text{depth}_{R}(\omega/xR)=d-1$.\newline\newline Since $\omega/xR$ is a Cohen-Macaulay $R$-module of codimension $1$, it follows from \cite[Theorem 3.3.10]{BH93} that $\left(\omega/xR \right)^{\vee}=\Ext_{R}^{1}(\omega/xR, \omega)$ is also a Cohen-Macaulay $R$-module of codimension $1$. Applying $\Hom_{R}(-,\omega)$ to sequence \ref{sequence:2.16} gives us the following exact sequence.\newline
\[
0 \longrightarrow \Ext_{R}^{1}(R/\omega, \omega) \longrightarrow \Ext_{R}^{1}(R/xR, \omega) \longrightarrow \Ext_{R}^{1}(\omega/xR, \omega) \longrightarrow 0
\]\text{}\newline Since $\Ext_{R}^{1}(R/\omega,\omega)$ is the canonical module for Gorenstein ring $R/\omega$, we have $\Ext_{R}^{1}(R/\omega,\omega) \cong R/\omega$. Similarly, we have $\Ext_{R}^{1}(R/xR,\omega) \cong \omega/x\omega$ \cite[Theorem 3.3.10, Proposition 3.3.18]{BH93}. This gives us the following exact sequence.
\begin{equation}\label{sequence:2.17}
0 \longrightarrow R/\omega \longrightarrow \omega/x\omega \longrightarrow \left(\omega/xR \right)^{\vee} \longrightarrow 0
\end{equation} Now consider the pullback diagram for sequence \ref{sequence:2.17} and the minimal MCM approximation of $\left(\omega/xR \right)^{\vee}$.
\begin{equation}\label{diagram:2.18}
\xymatrix{&& 0 \ar[d] & 0 \ar[d] \\
&& Y_{\left(\omega/xR \right)^{\vee}} \ar@{=}[r] \ar[d] & Y_{\left(\omega/xR \right)^{\vee}} \ar[d]\\  
0 \ar[r] &
R/\omega \ar@{=}[d] \,\ar[r] & Z \ar[r] \ar[d] & X_{\left(\omega/xR \right)^{\vee}} \ar[d] \ar[r]  & 0 \\
0 \ar[r] & R/\omega \ar[r] & \omega/x\omega \ar[r] \ar[d] & \left(\omega/xR \right)^{\vee} \ar[r] \ar[d]& 0 \\
&& 0 & 0 \\
} 
\end{equation} \text{}\newline Since $Y_{\left(\omega/xR \right)^{\vee}}$ and $\omega/x\omega$ have finite injective dimension, it follows that $Z$ also has finite injective dimension. Therefore, the middle row of diagram \ref{diagram:2.18}
\begin{equation}
0 \longrightarrow R/\omega \longrightarrow Z \longrightarrow X_{\left(\omega/xR \right)^{\vee}} \longrightarrow 0
\end{equation} is an FID hull of $R/\omega$, and $X_{\left(\omega/xR \right)^{\vee}} \cong_{\omega} X^{R/\omega}$. By Proposition \ref{prop:2.4} (b), we have $X^{R/\omega} \cong_{\omega} X^R$. Therefore, $X_{\left(\omega/xR \right)^{\vee}} \cong_{\omega} X^R$. By Proposition \ref{prop:2.7} (a), we have $X_{\omega/xR} \cong_{\omega} X^{R}$. This gives us part (c). Finally, the short exact sequence in part (d) is obtained by letting $M=R/\omega$ in diagram \ref{diagram:2.1} above.
\end{proof}
\section{$\text{SC}_r$-conditions for MCM modules}

In this section, $(R,\mathfrak{m})$ is a Gorenstein complete local ring. Using arguments from the proof of \cite[Theorem 2.2]{YI00}, we prove an inductive criterion for determining when an MCM $R$-module satisfies the $\SC_r$-condition. For $r>0$, we let CM$^{r}(R)$ denote the class of Cohen-Macaulay $R$-modules of codimension $r$. We let $\text{CM}(R)$ denote the class of MCM $R$-modules. Recall that two $R$-modules $M$ and $N$ are stably isomorphic if there are free $R$-modules $F$ and $G$ such that $M \oplus F \cong N \oplus G$. If $M$ and $N$ are stably isomorphic, we write $M \overset{st}{\cong} N$. For an $R$-module $M$ and a positive integer $i \geq 0$, we let $\Omega^{i}_{R}(M)$ denote the $i$th syzygy module of $M$.
\newline
\begin{definition}\cite[Definition 2.1]{KK07}\label{definition:3.1}\,
Let $R$ be a $d$-dimensional Cohen-Macaulay local ring with canonical module and let \,$0\leq r\leq d$. An MCM $R$-module $X$ satisfies the $\text{SC}_{r}$-{\it{condition}} if there is a finitely-generated $R$-module $M$ of codimension $r$ such that $X_M \overset{st}{\cong} X$. If every MCM $R$-module satisfies the $\text{SC}_{r}$-condition, we say that $R$ satisfies the SC$_r$-condition.\end{definition}
\text{}
\begin{proposition}\cite[Proposition 2.2]{KK07}\label{prop:3.2}
Let $R$ be a Gorenstein complete local ring and let $X$ be an MCM $R$-module. Let $r$ be a positive integer. The following are equivalent.\newline
\begin{enumerate}[label = (\arabic*)]
\item $X$ satisfies the $\SC_r$-condition; there is a finitely-generated $R$-module $M$ of codimension $r$ such that $X_M \overset{st}{\cong}X$.\newline 
\item There is a Cohen-Macaulay $R$-module $C$ of codimension $r$ such that $X_C \overset{st}{\cong} X$.
\end{enumerate}
\end{proposition}\text{}

The classes of rings satisfying the $\SC_{r}$-conditions are ordered by inclusion as follows.
\newline
\begin{proposition}\cite[Proposition 2.5]{KK07}\label{prop:3.3}
\,Let $R$ be a Gorenstein complete local ring and let $X$ be an MCM $R$-module. Let $r>0$. If $X$ satisfies the $\SC_{r+1}$-condition, then $X$ satisfies the $\SC_{r}$-condition. Therefore, if $R$ satisfies the $\SC_{r+1}$-condition, then $R$ satisfies the $\SC_{r}$-condition.
\end{proposition}
\text{}
\begin{remark}
Let $R$ be a Gorenstein complete local ring and let $X$ be an MCM $R$-module that satisfies the $\SC_{r}$-condition. By Proposition \ref{prop:3.2}, we have $X \overset{st}{\cong} X_C$, where $C$ is a Cohen-Macaulay $R$-module of codimension $r$. Let $C^{\vee}=\text{Ext}_{R}^{r}(C,R)$. Then we have $X_{C} \cong \text{Hom}_{R}(\Omega_{R}^{r}(C^{\vee}),R)$ and $X \overset{st}{\cong} \text{Hom}_{R}(\Omega_{R}^{r}(C^{\vee}),R)$ by \cite[Proposition 11.15]{LW12}.
\end{remark}\text{}

\begin{proposition}\cite[Proposition 2.5]{KK07}\label{prop:3.4}
Let $R$ be a Gorenstein complete local ring and let $r$ be a positive integer. If $R$ satisfies the $\SC_r$-condition, then $R_{\mathfrak{p}}$ is regular for each prime ideal $\mathfrak{p}$ of $R$ with $\Ht\mathfrak{p}<r$.
\end{proposition}
\text{}
\begin{definition}\cite[]{YI00}
Let $R$ be a Cohen-Macaulay local ring with canonical module $\omega$. Let\, $D_{R}(-):=\text{Hom}_{R}(-,\omega)$ and let $M$ be an MCM $R$-module. For each integer \,$i<0$, we define the $R$-module $\Omega_{R}^{i}(M)$ \,by\, $\Omega_{R}^{i}(M):=D_{R}(\Omega_{R}^{-i}(D_{R}(M))).$
\end{definition}
\text{}
\begin{lemma}\label{lemma:3.6}
Let $R$ be a $d$-dimensional Gorenstein complete local ring. Let $M$ be a Cohen-Macaulay $R$-module of codepth $r$. Then for any $r \leq t \leq d$, we have $X_M \overset{st}{\cong} \Omega_{R}^{-t}(\Omega_{R}^{t}(M))$.
\end{lemma}
\begin{proof}
Taking minimal projective resolutions of $Y_M$ and $M$, we apply the horseshoe lemma to the minimal MCM approximation of $M$,
\[
0 \longrightarrow Y_M \longrightarrow X_M \longrightarrow M \longrightarrow 0.
\] Since $M$ has codepth $r$, it follows that $\Omega_{R}^t(M)$ is an MCM $R$-module and $\Omega_{R}^{t}(Y_M)$ is an MCM $R$-module of finite projective dimension, and therefore free by \cite[Proposition 11.7]{LW12}. So $\Omega_{R}^{t}(M) \overset{st}{\cong} \Omega_{R}^{t}(X_M)$. The rest of the proof follows from \cite[Theorem 3.3.10]{BH93}.
\end{proof}\text{}
\begin{lemma}\cite[proof of Theorem 1.4]{YI00}\label{lemma:3.7}
Let ${\bf{x}} \in \mathfrak{m}$ be an $R$-regular sequence, and let $M$ be a finitely-generated $R/{\bf{x}}R$-module. For $n \geq 0$, we have \,$\Omega_R^{n+1}(M) \overset{st}{\cong} \Omega_R^n(\Omega_{R/{\bf{x}}R}^1(M))$.
\end{lemma}\text{}
\begin{remark}
The main result of this section-Proposition \ref{prop:3.12}-builds on the following results. 
\end{remark}\text{}
\begin{proposition}\cite[Corollary 11.23]{LW12}\label{prop:3.9}
Let $R$ be a Cohen-Macaulay local ring with canonical module. Assume $R$ is generically Gorenstein. Then the following statements are equivalent.\newline
\begin{enumerate}[label = (\arabic*)]
\item $R$ is a domain. \newline 
\item $R$ satisfies the $\SC_{1}$-condition.
\end{enumerate}
\end{proposition}
\begin{proposition}\cite[Theorem 2.2]{YI00} 
Let $R$ be a normal Gorenstein complete local ring of dimension two. Then the following are equivalent.
\begin{enumerate}[label = (\arabic*)]
\item $R$ is a UFD. \newline
\item For any MCM $R$-module, there is an $R$-module $L$ of finite length (hence a Cohen-Macaulay $R$-module of codimension two) such that $M \overset{st}{\cong} \Omega_{R}^{2}(L)$. \newline
\item $R$ satisfies the $\SC_2$-condition.
\end{enumerate}
\end{proposition}\text{}
\begin{theorem}\cite[Theorem 2.9]{KK07}\label{theorem:3.11} A Gorenstein complete local ring $R$ satisfies the $\SC_2$-condition if and only if $R$ is a UFD.
\end{theorem}\text{}
\begin{proposition}\label{prop:3.12}
Let $R$ be a Gorenstein complete local ring of dimension $d \geq 3$\, and \,$3 \leq r \leq d$. Let $M$ be an MCM $R$-module and suppose $\Omega_{R}^{1}(M)$ satisfies the $\SC_{r-1}$-condition. Let $L \in \CM^{r-1}(R)$ such that $X_{L} \stackrel{st}\cong \Omega_{R}^1(M)$. If there is a regular sequence \,${\bf{x}} \in \Ann_{R}(L)$\, of length $r-2$ such that $R/{\bf{x}}R$ is a UFD, then $M$ satisfies the $\SC_{r}$-condition. 
\end{proposition}
\begin{proof}
We first prove that $\Omega_{R}^{r-1}(L) \overset{st}{\cong} \Omega_{R}^{r}(M)$. Taking minimal projective resolutions of $L$ and $Y_L$, we apply the horseshoe lemma to the minimal MCM approximation of $L$ and obtain the following diagram with exact rows, and columns that are truncated projective resolutions.
\begin{equation}\label{diagram:3.1}
\begin{tikzcd}
&0 \rar &\Omega_{R}^{r-1}(Y_L) \dar \rar &Z_{r-1} \rar \dar &\Omega_{R}^{r-1}(L) \dar \rar &0\\
&0 \rar &P_{r-1}' \rar \dar &P_{r-1} \rar \dar &P_{r-1}'' \rar \dar{} &0\\
&&\vdots \dar &\vdots \dar &\vdots \dar &\\
&0 \rar &P_0' \rar \dar &P_0 \rar \dar &P_0'' \dar \rar &0\\
&0 \rar &Y_{L} \dar \rar &X_{L} \dar \rar &L \dar \rar &0\\
&& 0 & 0 & 0 &&
\end{tikzcd}
\end{equation}
By the depth lemma, we have depth$_{R}(Y_L) \geq \text{min}\{\text{depth}_{R}(X_{L}), \text{depth}_{R}(L)+1\}$. Since $X_L$ is MCM \,and $\text{depth}_{R}(L)=d-(r-1)$, we have $\text{depth}_{R}(Y_L) > d-(r-1)$. By successively applying the depth lemma, it follows that $\Omega_{R}^{r-1}(Y_L)$ is an MCM $R$-module  of finite projective dimension. Therefore, $\Omega_{R}^{r-1}(Y_L)$ is a free $R$-module by \cite[Proposition 11.7]{LW12}. Likewise, $\Omega_{R}^{r-1}(L)$ is an MCM $R$-module. Applying $\text{Hom}_{R}(\Omega_{R}^{r-1}(L),-)$ to the top row of diagram \ref{diagram:3.1}, we obtain the following exact sequence.
\[
\text{Hom}_{R}(\Omega_{R}^{r-1}(L)\,,\,Z_{r-1}) \longrightarrow \text{Hom}_{R}(\Omega_{R}^{r-1}(L)\,,\,\Omega_{R}^{r-1}(L)) \longrightarrow \text{Ext}_{R}^1(\Omega_{R}^{r-1}(L)\,,\,\Omega_{R}^{r-1}(Y_{L}))\] Since $\Omega_{R}^{r-1}(L)$ is MCM and $\Omega_{R}^{r-1}(Y_{L})$ has finite projective dimension, we have 
\[\text{Ext}_{R}^1(\Omega_{R}^{r-1}(L)\,,\,\Omega_{R}^{r-1}(Y_{L}))=0\] by \cite[Proposition 11.3]{LW12}. Therefore, the top row of the diagram \ref{diagram:3.1} splits, and we have $\Omega_{R}^{r-1}(X_{L}) \overset{st}{\cong} Z_{r-1} \overset{st}{\cong} \Omega_{R}^{r-1}(L).$ Since $X_{L} \overset{st}{\cong}\Omega_{R}^{1}(M)$, it follows that $\Omega_{R}^{r-1}(L) \overset{st}{\cong} \Omega_{R}^{r}(M)$. \newline

Let $S:=R/{\bf{x}}R$. Then $L \in $ CM$^{1}(S)$. We denote the associated primes of $L$ by $\text{Ass}(L)=\{\mathfrak{p}_{1},..., \mathfrak{p}_{m}\}$ and we let \,$\Gamma:=S-\bigcup\limits_{i=1}^{m} \mathfrak{p}_{i}$. Since $L$ is a Cohen-Macaulay $S$-module of codimension 1, we have \,$\text{ht\,}\mathfrak{p}_{i}=1$ for each $i$ by \cite[Theorems 17.3 and 17.4]{M89}. Since $S$ is a UFD,\, $\mathfrak{p}_{i}$ is a principal ideal for each $i$ \cite[Lemma 2.2.17]{BH93}. Write\, $\mathfrak{p}_{i}=(p_i)$, where $p_i \in S$. Let $\mathfrak{q}$ be a nonzero prime ideal of $\Gamma^{-1}S$ and let $h:S \longrightarrow \Gamma^{-1}S$ be the localization map. Then\, $h^{-1}(\mathfrak{q}) \subseteq \bigcup\limits_{i=1}^{m} (p_{i})$ is a nonzero prime ideal, and by prime avoidance, $h^{-1}(\mathfrak{q}) \subseteq (p_{i})$ for some $(p_i)$.
Therefore, $h^{-1}(\mathfrak{q})=(p_{i})$ and $\mathfrak{q}=p_{i}\Gamma^{-1}S$. We conclude that every prime ideal of $\Gamma^{-1}S$ is principal. Therefore, $\Gamma^{-1}S$ is a PID. Since \,$\Gamma^{-1}L$ is a finitely generated \,$\Gamma^{-1}S$-module, there are elements $a_1,...,a_t \in \Gamma^{-1}S$ such that $\Gamma^{-1}L \cong \bigoplus\limits_{k=1}^{t} \Gamma^{-1}S/a_{k}\Gamma^{-1}S$ \,as \,$\Gamma^{-1}S$-modules. Let \,$\phi: \Gamma^{-1}L \longrightarrow \bigoplus\limits_{k=1}^{t} \Gamma^{-1}S/a_{k}\Gamma^{-1}S$\, be the corresponding isomorphism. Since $L \in \text{CM}^1(S)$,\, it follows that $L$ is a torsion $S$-\text{module, and } $\Gamma^{-1}L$\, is a torsion \,$\Gamma^{-1}S$-module. Therefore, $a_{k} \neq 0$ for all $k$. \newline

Fix $k$. We claim that each associated prime ideal of $I=h^{-1}(a_k\Gamma^{-1}S)$ \,has height one. Since $I \neq 0$, the associated primes of $I$ have height at least one. Suppose $I$ has an associated prime $\mathfrak{q}$ of height greater than one.  Then $\Gamma \cap \mathfrak{q} \neq \emptyset$. Let $s \in \Gamma \cap \mathfrak{q}$. Since $\mathfrak{q}$ is an associated prime of $I$, there exists an element $x \in S \setminus I$ such that $\mathfrak{q}=\Ann_{S}(\overline{x})$, with $\overline{x} \in S/I$. Therefore, $\mathfrak{q}x \subseteq I$ and $sx \in I$. It follows that $\frac{sx}{1} \in a_{k}\Gamma^{-1}S$,\, so\, $\frac{x}{1} \in a_{k}\Gamma^{-1}S$ \,and \,$x \in I$, which is a contradiction. We conclude that every associated prime of $I$ has height one, and is therefore principal.
Let $\text{Ass}(I)=\{(q_1),...,(q_m)\}$. Taking an irredundant primary decomposition of $I$, we have \[I=Q_{1} \cap Q_{2} \cap ... \cap Q_{v},\] where for each $l$, $\sqrt{Q_l}=(q_l)$. We claim that each $Q_l$ is a principal ideal. Fix $1 \leq l \leq v$. We have $Q_l \subseteq (q_l)$. Since $\bigcap\limits_{\alpha \geq 0} (q_{l}^\alpha)=0$, there is a positive integer $\alpha$ such that $Q_l \subseteq (q_{l}^\alpha)$ and $Q_l \not\subseteq (q_{l}^{\alpha+1})$. Since $S$ is Noetherian, $Q_l$ is finitely-generated. Write $Q_l=(y_1,...,y_r)$, where $y_1,...,y_r \in S$. Since $Q_l \subseteq (q_{l}^\alpha)$, there are elements $c_1,...,c_r$ in $S$ such that $y_j=c_jq_l^{\alpha}$ for $j=1,...,r$. So\, $Q_l=(c_1q_{l}^{\alpha},...,c_rq_{l}^{\alpha})$. Since $Q_{l} \not\subseteq (q_{l}^{\alpha+1})$, there is an index $j$ such that $c_j \not\in (q_l)$.
Since $c_jq_{l}^\alpha \in Q_l$ and $Q_l$ is a primary ideal, $q_{l}^\alpha \in Q_l$ or $c_{j}^{\beta} \in Q_l$ for some $\beta>0$. If $c_{j}^{\beta} \in Q_l$ for some $\beta>0$, then $c_j \in \sqrt{Q_l}=(q_l)$, which is false. Therefore, $q_{l}^\alpha \in Q_l \subseteq (q_l^{\alpha})$, whence $Q_l=(q_{l}^\alpha)$. We conclude that $Q_l$ is a principal ideal for $l=1,...,v$. Therefore, $I$ is an intersection of principal ideals, and since $S$ is a UFD, $I$ is also a principal ideal. \newline

Let $b_k \in S$ such that $I=(b_k)$. Then \,$a_k\Gamma^{-1}S=b_k\Gamma^{-1}S$. We claim that Ass$(b_k)=\text{Ass}(I) \subseteq$ Ass$(L)=\{(p_1),...,(p_m)\}$. Suppose $\mathfrak{q} \in \text{Ass}(I)-\text{Ass}(L)$. Since $\mathfrak{q}$ has height 1, $\mathfrak{q} \cap \Gamma \neq \emptyset$. Let $s \in \mathfrak{q} \cap \Gamma $, and let $x \in S-I$ such that $\mathfrak{q}=\Ann_{S}(\overline{x})$, with $\overline{x} \in S/I$. Then $sx \in I$, so $\frac{sx}{1}\in a_k\Gamma^{-1}S$ and $x \in I$, a contradiction. We conclude that Ass$(b_k) \subseteq$ Ass$(L)$.
We may therefore assume that $a_k \in S$ and that every associated prime of $a_k$ is an associated prime of $L$. Since there is an isomorphism of $\Gamma^{-1}S-$modules
\[\Gamma^{-1}\text{Hom}_{S}(L,\bigoplus\limits_{k=1}^{t}S/a_{k}S) \cong \text{Hom}_{\Gamma^{-1}S}(\Gamma^{-1}L,\bigoplus\limits_{k=1}^{t} \Gamma^{-1}S/a_{k}\Gamma^{-1}S),\]
there is an $S$-map $f:L \longrightarrow \bigoplus\limits_{k=1}^{t}S/a_{k}S$ \,such that $\Gamma^{-1}f=\phi$. We claim that $f$ is a monomorphism. Suppose not. Then $\ker f \neq 0$, and Ass$(\ker f) \neq \emptyset$. Let $\mathfrak{p} \in$ \text{Ass}$(\ker f)$. Since Ass$(\ker f) \subseteq$ Supp$(\ker f)$, we have $(\ker f)_{\mathfrak{p}} \neq 0$. Also, since $\ker f \subseteq L$, we have $\mathfrak{p} \in$ Ass\,$L=\{\mathfrak{p}_{1},...,\mathfrak{p}_{m}\}$. Let $T=S-\mathfrak{p}$ and let $T'$ be the image of $T$ in $\Gamma^{-1}S$. Since $\Gamma=S-\bigcup\limits_{i=1}^{m}\mathfrak{p}_{i}$, we have $\Gamma \subseteq T$. Starting with our \,$\Gamma^{-1}S$-isomorphism $\Gamma^{-1}f: \Gamma^{-1}L \longrightarrow \bigoplus\limits_{k=1}^{t} \Gamma^{-1}S/a_{k}\Gamma^{-1}S$, we localize at $T'$, obtaining the $T^{-1}S$-isomorphism \[T^{-1}f:T^{-1}L \longrightarrow \bigoplus\limits_{k=1}^{t} T^{-1}S/a_{k}T^{-1}S.\] But this map is just
$f_{\mathfrak{p}}: L_{\mathfrak{p}} \rightarrow \bigoplus\limits_{k=1}^{t} S_{\mathfrak{p}}/a_{k}S_{\mathfrak{p}}$. Therefore, we have $0=\ker (f_{\mathfrak{p}})=(\ker f)_{\mathfrak{p}} \neq 0$, a contradiction.
Therefore, $\ker f=0$ and we have an exact sequence of $S$-modules
\begin{equation}\label{sequence:3.2} 0 \longrightarrow L \xlongrightarrow{f} \bigoplus\limits_{k=1}^{t} S/a_{k}S  \longrightarrow L' \longrightarrow 0.
\end{equation}
Taking minimal projective resolutions of $L$ and $L'$, the Horseshoe Lemma gives us the following diagram with exact rows, and columns that are truncated projective resolutions.\[\xymatrix{
0 \ar[r] & \Omega_{S}^{1}(L) \ar[r] \ar[d] & Z_1 \ar[r] \ar[d] & \Omega_{S}^{1}(L') \ar[r] \ar[d] & 0 \\
0 \ar[r] & P_{0} \ar[r] \ar[d] & P_{0} \bigoplus P_{0}' \ar[r] \ar[d] & P_{0}' \ar[r] \ar[d] & 0 \\
0 \ar[r] & L \ar[r] \ar[d] & \bigoplus\limits_{k=1}^{t} S/a_{k}S \ar[r] \ar[d] & L' \ar[r] \ar[d] & 0 \\
  & 0 & 0 & 0} 
\]
Since $\text{pd}_{S}S/a_{k}S=1$ for each $k$, it follows that $Z_1$ is a free $S$-module. Therefore, $\Omega_{S}^2(L') \overset{st}{\cong} \Omega_{S}^1(L).$ Since $\Omega_{R}^{r-1}(L) \overset{st}{\cong} \Omega_{R}^{r}(M)$, by Lemma \ref{lemma:3.7}, we have 
\[\Omega_{R}^{r-2}(\Omega_{S}^2(L')) \overset{st}{\cong} \Omega_{R}^{r-2}(\Omega_{S}^{1}(L))\overset{st}{\cong} \Omega_{R}^{r-1}(L) \overset{st}{\cong} \Omega_{R}^{r}(M).\] On the other hand, \[\Omega_{R}^{r-2}(\Omega_{S}^2(L')) \overset{st}{\cong} \Omega_{R}^{r-1}(\Omega_{S}^{1}(L')) \overset{st}{\cong} \Omega_{R}^{r}(L').\] 
Therefore, $\Omega_{R}^{r}(L') \overset{st}{\cong} \Omega_{R}^{r}(M)$. We claim that $L'$ is Cohen-Macaulay. We have \[\text{Supp}(L') \subseteq \text{Supp}(\bigoplus\limits_{k=1}^{t} S/a_{k}S) \subseteq \bigcup\limits_{k=1}^{t} \text{Supp}(S/a_kS).\] 
Applying the depth lemma to sequence \ref{sequence:3.2}, we have $\text{depth}_{S}(L') \geq \text{depth}(S)-2=d-r$. Since \[\Dim_{S} (\bigoplus\limits_{k=1}^{t} S/a_{k}S)= \Dim(S)-1=d-r+1,\] we have \,$\text{dim}_{S}(L') \leq d-r+1$.
Suppose \,$\text{dim}_{S}(L') = d-r+1$. Then there is a chain of prime ideals $\mathfrak{q}_{0} \subsetneq \mathfrak{q}_{1} \subsetneq \cdot\cdot\cdot \subsetneq \mathfrak{q}_{d-r+1}$ in Supp$(L')$. Since $\mathfrak{\mathfrak{q}}_{0} \in \text{Supp}(L')$, there exists an index $k$ such that $\mathfrak{\mathfrak{q}}_{0} \in  \text{Supp}(S/a_kS)$. Therefore, $\mathfrak{q}_{0} \subsetneq \mathfrak{q}_{1} \subsetneq \cdot\cdot\cdot \subsetneq \mathfrak{q}_{d-r+1}$ is a chain of prime ideals in Supp$(S/a_kS)$. Since the Krull dimension of Supp($S/a_kS)$ is $d-r+1$, it follows that $\mathfrak{q}_{0}$ is a minimal prime in Supp$(S/a_kS)$. Therefore, $\mathfrak{q}_{0}=\mathfrak{p}_{i} \in \Ass(L)$ for some $1 \leq i \leq m$. Since $f_{\mathfrak{p}_{i}}: L_{\mathfrak{p}_{i}} \longrightarrow \bigoplus\limits_{k=1}^{t} S_{\mathfrak{p}_{i}}/a_{k}S_{\mathfrak{p}_{i}}$ is an isomorphism, we have $0 \neq L'_{\mathfrak{q}_{0}}=L'_{\mathfrak{p}_{i}}=0$, a contradiction. We conclude that dim$_{S}(L')=d-r$ and \[L' \in \CM^{2}(S) \subseteq \,\CM^{r}(R).\] Since $\Omega_{R}^{r}(L') \overset{st}{\cong} \Omega_{R}^{r}(M)$, by Lemma \ref{lemma:3.6}, we have 
\[X_{L'} \overset{st}{\cong} \Omega_{R}^{-r}(\Omega_{R}^{r}(L')) \overset{st}{\cong} \Omega_{R}^{-r}(\Omega_{R}^{r}(M)) \overset{st}{\cong} M.\].
\end{proof}\text{}
Let $\text{Spec}(R)$ denote the set of prime ideals of $R$ and let \,$U_{R}:=\text{Spec}(R) \setminus \{\mathfrak{m}\}$ denote the punctured spectrum of $R$. Let \,$\text{Pic}(U_{R})$ denote the Picard group of $U_{R}$ \cite[Chapter 5]{Fossum73}.\newline
\begin{definition}\cite[Chapter 5]{Fossum73}
A Noetherian local ring $R$ is {\it{parafactorial}} if $\text{depth}(R) \geq 2$ and $\text{Pic}(U_{R})=0$.
\end{definition} 
\text{}
\begin{remark}
This weaker notion of factoriality gives us the following criterion for determining when a ring is a UFD. We use this criterion and  Kato's result on regular localizations of rings satisfying the SC$_r$-condition (Proposition \ref{prop:3.4}) to study the relation between the SC$_r$-condition and UFDs obtained by factoring out a regular sequence. 
\end{remark}\text{}
\begin{proposition}\cite[Corollary 18.11]{Fossum73}\label{prop:3.15}
Suppose $R$ is a Noetherian local ring such that $\Dim(R) \geq 2$. Then $R$ is a UFD if and only if $R$ is parafactorial and $R_{\mathfrak{p}}$ is a UFD for all $\mathfrak{p} \in U_{R}$.
\end{proposition}\text{}
\begin{corollary}\label{corollary:3.16} Let $(R,\mathfrak{m})$ be a Gorenstein complete local ring of dimension $d \geq 3$. The following are equivalent.\newline
\begin{enumerate}[label = (\arabic*)]
 \item $R$ satisfies the $\SC_{d-1}$-condition and for each $M \in \CM(R$), there is a module $L \in \CM^{d-1}(R)$ such that $X_{L} \stackrel{\text{st}}\cong \Omega_{R}^1(M)$ and $\Ann_{R}(L)$ contains a regular sequence $\bold{x}$ of length $d-2$ such that $R/\bold{x}R$ is a UFD.\newline\newline
\item $R$ satisfies the $\SC_{d}$-condition and for each $M \in \CM(R$), there is a module $L \in \CM^{d-1}(R)$ such that $X_{L} \stackrel{\text{st}}\cong \Omega_{R}^1(M)$ and $\Ann_{R}(L)$ contains a regular sequence $\bf{x}$ of length $d-2$ that satisfies the following.\newline
\begin{itemize}
\item[(i)] $\bf{x}$ is a subset of a regular system of parameters for \,$R_{\mathfrak{p}}$\, for all prime ideals $\mathfrak{p} \in U_{R}$ containing $\bf{x}$\newline
 \item[(ii)] $\Pic(U_{R/{\bf{x}}R})=0$
 \end{itemize}
\end{enumerate}
\end{corollary}\text{}
\begin{proof} $(1) \Rightarrow (2).$ By Proposition \ref{prop:3.12}, $R$ satisfies the SC$_{d}$-condition. Let $M \in$ CM$(R)$, and let\, $L$ and \,${\bf{x}}$\, be as in $(1)$. Let $\mathfrak{p} \in U_{R}$ be a prime ideal that contains $\bf{x}$. Let $\pi:R\longrightarrow R/{\bf{x}}R$\, be the quotient map, and let $\mathfrak{P}=\pi(\mathfrak{p})$. Then $\mathfrak{P}$ is a prime ideal and $\Ht(\mathfrak{P})<2$. Since $\Ht(\mathfrak{p})<d$ and $R$ satisfies the SC$_{d}$-condition, the localization $R_{\mathfrak{p}}$ is regular by Proposition \ref{prop:3.4}. Since $R/{\bf{x}}R$ is a Gorenstein complete local UFD, it satisfies the SC$_{2}$-condition by Theorem \ref{theorem:3.11}. By Proposition \ref{prop:3.4}, we have $(R/{\bf{x}}R)_{\mathfrak{P}} \cong R_{\mathfrak{p}}/{\bf{x}}R_{\mathfrak{p}}$ is a regular local ring. Thus, $\bf{x}$ is a subset of a regular system of parameters for $R_{\mathfrak{p}}$ \cite[Proposition 2.2.4]{BH93}. Finally, we have Pic$(U_{R/{\bf{x}}R})=0$\, by Proposition \ref{prop:3.15}. \newline\newline\newline
$(2) \Rightarrow (1)$ Since $R$ satisfies the SC$_{d}$-condition, $R$ satisfies the SC$_{d-1}$-condition by Proposition \ref{prop:3.3}. Let $M \in {\text{CM}}(R)$, and let $L$ and ${\bf{x}}$ be as in $(2)$. We prove that $R/{\bf{x}}R$ is a UFD. Let $\mathfrak{P} \in U_{R/{\bf{x}}R}$ and let $\mathfrak{p}=\pi^{-1}(\mathfrak{P})$. Then $\mathfrak{p}$ is a prime ideal containing $\bf{x}$ and $\Ht(\mathfrak{p})<d$. Since $R$ satisfies the SC$_{d}$-condition, $R_{\mathfrak{p}}$ is a regular local ring. Since $\bf{x}$ is a subset of a regular system of parameters for $R_{\mathfrak{p}}$, we have $R_{\mathfrak{p}}/{\bf{x}}R_{\mathfrak{p}} \cong (R/{\bf{x}}R)_{\mathfrak{P}}$ is a regular local ring, and therefore a UFD \cite[Proposition 2.2.4 and Theorem 2.2.19]{BH93}. Therefore, $R/{\bf{x}}R$ is a UFD by Proposition \ref{prop:3.15}.
\end{proof}\text{}
\begin{corollary}
Let $R$ be a Gorenstein complete local ring of dimension $3$. Assume that $R$ is a UFD and for each $M \in \CM(R)$ there exists $L \in \CM^{2}(R)$ such that $X_{L} \overset{st}{\cong} \Omega_{R}^{1}(M)$ and a nonzerodivisor $x \in \Ann_{R}(L)$ such that $R/xR$ is also a UFD. Then $R$ satisfies the $\SC_{3}$-condition.
\end{corollary}\text{}
\section*{Acknowledgements} Section 3 of this paper is from my PhD dissertation. I am thankful to my PhD advisor Graham Leuschke for his helpful suggestions and guidance. 
\text{}\newline

\bibliographystyle{amsplain}
}
\end{document}